\documentclass[graybox]{svmult}

\usepackage{_Definitions}
\usepackage{_ref}

\makeindex             

\begin{document}

\title*{Persistence module and Schubert calculus}
\author{Yasuaki Hiraoka, Kohei Yahiro, Chenguang Xu}
\institute{Yasuaki Hiraoka \at Kyoto University Institute for Advanced Study, Kyoto University. Yoshida Ushinomiya-cho, Sakyo-ku, Kyoto 606-8501, Japan. \email{hiraoka.yasuaki.6z@kyoto-u.ac.jp}
\and Kohei Yahiro \at Kyoto University Institute for Advanced Study, Kyoto University. Yoshida Ushinomiya-cho, Sakyo-ku, Kyoto 606-8501, Japan. \email{yahiro.kohei.5r@kyoto-u.ac.jp}
\and Chenguang Xu \at Kyoto University Institute for Advanced Study, Kyoto University. Yoshida Ushinomiya-cho, Sakyo-ku, Kyoto 606-8501, Japan. \email{xu.chenguang.k34@kyoto-u.jp}
}

\maketitle

\abstract{
A multiplication on persistence diagrams is introduced by means of Schubert calculus. The key observation behind this multiplication comes from the fact that the representation space of persistence modules has the structure of the Schubert decomposition of a flag. In particular, isomorphism classes of persistence modules correspond to Schubert cells, thereby the Schubert calculus naturally defines a multiplication on persistence diagrams. 
The meaning of the multiplication on persistence diagrams is carried over from that on Schubert calculus, i.e., algebro-geometric intersections of varieties of persistence modules. 
}

\section{Introduction}
\label{sec:introduction}
Persistence module \cite{ELZ} is one of the most established methods in topological data analysis \cite{carlsson}, and its mathematical structures have been studied in detail. For example, the interval decomposition of persistence modules \cite{ZC} and its stability property \cite{CEH} provide us with reliable data descriptors called persistence diagrams, which have been applied in various scientific research \cite{PNAS_carlsson, PNAS_chad, PNAS_hiraoka, NC_hiraoka}. Furthermore, vectorizations of persistence diagrams \cite{landscape, PWGK1, PWGK2, PSSK} provide us with additive operations and metrics on persistence diagrams. This theoretical extension clarifies several important geometric and probabilistic aspects of persistence diagrams \cite{HST}, and accordingly, they are utilized in machine learning applications. 

From this historical background, it is natural to ask next whether we can introduce a proper multiplication operation on persistence diagrams. In this paper, we show that Schubert calculus \cite{fulton,ikeda} can be well-defined on persistence diagrams and provides a multiplication on them.
Here, we remark that Schubert calculus is studied as intersection theory of flag varieties in algebraic geometry, and there are many combinatorial research including Young diagrams,  Schubert polynomials, and so on. 

The key to introducing Schubert calculus on persistence diagrams comes from a flag structure equipped in persistence modules. In the setting of graded modules \cite{knudson, ZC}, we know that allowable basis change operation should respect the birth times of generators, and this rule restricts basis change matrices to be (upper) triangular. It follows from this fact that a natural structure of flags in a vector space will be induced in persistence modules. From this correspondence, we see that isomorphism classes of persistence modules can be equivalently regarded as Schubert cells, and accordingly, we have a multiplication on persistence diagrams induced by that on Schubert varieties. 
Furthermore, the meaning of the multiplication on persistence diagrams is carried over from that on Schubert calculus, i.e., algebro-geometric intersections of varieties of persistence modules. 

This paper is organized as follows. Section \ref{sec:decomposition} summarizes some preliminaries about flag varieties and the Schubert decomposition, and then shows the correspondence between isomorphism classes of persistence modules and Schubert cells. In Section \ref{sec:calculus}, we show a multiplication on persistence diagrams based on Schubert calculus, as well as a connection to Schubert polynomials. Several examples of multiplications on persistence diagrams are also shown there. 

\section{Persistence module and Schubert decomposition}
\label{sec:decomposition}

\subsection{Flag variety}
\label{sec:flag}
Let $E=\C^n$  be an $n$-dimensional ambient $\C$-vector space. A (complete) flag is a sequence of vector subspaces
\[
V_{\bullet}: \{\mathbf{0}\}=V_0\subset V_1\subset \dots \subset V_n=E,\quad \dim V_i=i. 
\]
As a dual notation, a flag can also be expressed as 
\[
V^{\bullet}: E=V^0\supset V^1\supset \dots \supset V^n=\{\mathbf{0}\},\quad {\rm codim~} V^i=i.
\]
The set of all flags in $E$ is  denoted by 
$\cF_n$, called the flag variety. We can construct a surjective map
\[
\begin{aligned}
\pi: \GL_n(\C)&\to  \cF_n\\
({\bm v}_1 \dots {\bm v}_n)&\mapsto (V_i:=\langle{{\bm v}_1,\dots,{\bm v}_i\rangle})_{i=1}^n,
\end{aligned}
\]
where $\GL_n(\C)$ is the general linear group, $({\bm v}_1 \dots {\bm v}_n)$ is an $n\times n$ non-singular matrix expressed by its column vectors, and $\langle{\bm v}_1,\dots,{\bm v}_i\rangle$ is the subspace spanned by $\{{\bm v}_1,\dots,{\bm v}_i\}$. Also, it is easy to see that $\GL_n(\C)$ naturally acts on $\cF_n$, and this action makes $\cF_n$  a homogeneous space. Hence, we have $\GL_n(\C)/B\simeq \cF_n$, where $B$ is the subset of upper triangular matrices.
This bijection provides $\cF_n$ with a structure of a projective variety \cite{humphreys}.

\subsection{Schubert decomposition}
\label{sec:decompositon}

\begin{defn}
Let $S_n$ denote the symmetric group on $n$ elements. For $w\in S_n$, we define

\begin{enumerate}[(i)]
    \item  $V^w_\bullet\in\cF_n$ to be a flag associated with $w$ given by
\[
V^w_i:=\langle \bm{e}_{w(1)},\dots,\bm{e}_{w(i)}\rangle, 
\quad 1\leq i\leq n,
\]
    \item  $F^\bullet_w$ to be the flag complement to $V^w_\bullet$ indexed by codimension:
\[
F^i_w:=\langle\bm{e}_{w(i+1)},\dots,\bm{e}_{w(n)}\rangle,\quad 0\leq i\leq n-1,
\]
    \item  $\mathscr{U}_w$ to be  the set of flags in $\cF_n$ in general position with $F_w^{\bullet}$:
\[
\mathscr{U}_w:=\{V_\bullet\in\cF_n~:~ F^i_w\cap V_i=\{0\}, 1\leq i\leq n\},
\]
    \item  $U_w\subseteq \GL_n(\C)$ to be the set of matrices satisfying the following conditions:
\[
U_w=\{\xi\in\GL_n(\C)~:~\xi_{w(j),j}=1\;\text{for}\;1\leq j\leq n\;\text{and}\; \xi_{i,k}=0\;\text{for}\;1\leq i\leq n,\;k>w^{-1}(i)\},
\]
\item 
\[
U_w^+:= \{\xi\in U_w~:~\xi_{jk}=0,\; j>w(k)\},
\]
\[
U_w^-:= \{\xi\in U_w~:~\xi_{jk}=0,\; j<w(k)\}.
\]
\end{enumerate}
\end{defn}

\begin{defn}
For $w\in S_n$, the images  
$X_w^\circ\coloneqq\pi(U_w^+)$ and $\Omega^\circ_{w}\coloneqq\pi(U_w^-)$
are called the \textit{Schubert cell} and \textit{opposite Schubert cell}, respectively. 
\end{defn}

Then, we have the following decomposition of the flag variety $\cF_n$.
\begin{thm}[\cite{fulton, ikeda}]\label{thm:cell-decomp}
The flag variety can be decomposed in the following two ways:
\begin{equation}\label{eq:schubert_decomposition}
\cF_n=\bigsqcup_{w\in S_n}X_w^\circ=\bigsqcup_{w\in S_n}\Omega^\circ_{w}.
\end{equation}
Both are called the Schubert cell decomposition of $\cF_n$. 
\end{thm}

For $w\in S_n$, let us define $d_w(i,j)=\#\{s\leq j~:~w(s)\geq i\}$.
Then, the Bruhat order on $S_n$ is given by 
\[
v\leq w \xLeftrightarrow{{\rm def}} d_v(i,j)\leq d_w(i,j)\quad{\rm for}~\forall i,j
\]
for $v,w\in S_n$.

\begin{defn}\label{defn:sch-var}
For $w\in S_n$, $X_w=\bigsqcup_{u\leq w} X^\circ_u$ and $\Omega_{w}=\bigsqcup_{u\geq w}\Omega_u^\circ$ are called the Schubert variety and the opposite Schubert variety, respectively. 
\end{defn}
It is known that those varieties are realized as the Zariski closure 
$X_w=\overline{X_w^{\circ}}$,
and $\Omega_w=\overline{\Omega_w^\circ}$.
\subsection{Correspondence between Schubert cells and persistence modules}
Let $\C[z]$ be the polynomial ring with a single variable $z$, and ${\rm mod}(\C[z])$ be the category of non-negatively graded $\C[z]$-modules. In this paper, we take the following definition of persistence modules. 
\begin{defn}
A (one-parameter) persistence module $M$ is defined as a finitely generated object in ${\rm mod}(\C[z])$.
\end{defn}
\begin{rem}{\rm (\cite{ZC})}
In this setting, the interval decomposition
\[
M\simeq \left(\bigoplus_{i=1}^p
(z^{b_i})/(z^{d_i})
\right)\oplus 
\left(\bigoplus_{i=p+1}^{p+q}
(z^{b_i})
\right)
\]
of the persistence module $M$ is derived by the structure theorem for finitely generated graded modules over $\C[z]$ (PID). Here, $(z^b)$ denotes the ideal generated by a monomial $z^b$. The multiset $D=\{[b_i,d_i)~:~i=1,\dots,p\}\sqcup\{[b_i,\infty)~:~i=p+1,\dots,p+q\}$ of intervals in $\R$ is called the persistence diagram or barcode. For each interval, $b_i$ and $d_i$ are called the birth time and death time, respectively. 
\end{rem}

For a persistence module $M$, we can construct a minimal free resolution 
\begin{equation}\label{eq:MFR}
\begin{tikzcd}
0\arrow[r]& F_1 \arrow[r,"\partial_1"]& F_0 \arrow[r,"\epsilon"] & M\arrow[r]& 0,
\end{tikzcd}
\end{equation}
with some free modules $F_0,F_1\in {\rm mod}(\C[z])$. 
Therefore, we have
\[
M\simeq F_0/\ker\epsilon=F_0/\image\partial_1,
\]
implying that the parametrization on isomorphism classes is given by the embedding $F_1\xrightarrow{\partial_1}F_0$. Hereafter, we regard $F_1$ as a submodule in $F_0$. 
Recall that the free modules $F_0$ and $F_1$ take the  forms
\begin{equation}\label{eq:F0F1}
	F_0\cong \bigoplus_{i=1}^m z^{b_i}\C[z],
	\quad
	F_1\cong \bigoplus_{i=1}^n z^{d_i}\C[z],	
\end{equation}
where $\bm{b}: b_1\leq \dots\leq b_m$ and $\bm{d}: d_1\leq\dots\leq d_n$ are birth and death sequences, respectively.
They satisfy the following conditions
\begin{enumerate}[(i)]\label{enum:restriction}
    \item $m\geq n$: any generator of $M$ needs to be born before the death, but can also persist to infinity;
    \item $b_i\leq d_i$ for $i=1,\dots,n$: the number of existing death generators shall not surpass that of birth generators.
\end{enumerate}

If $m>n$, by adding death generators distant enough, the number of two types of generators can be equalized without losing the essential structure. So we only consider the case $m=n$ (i.e., $\rank M=0$). 
Furthermore, for the sake of simplicity, we assume the following setting.
\begin{ass}\label{ass:bd_assumption}
The birth and death sequence is 
strict ($b_1< \dots< b_n, d_1<\dots< d_n$), and
separated ($b_n<d_1$).
\end{ass}
In fact, we can obtain a similar result without this assumption by replacing $B$ with a general parabolic subgroup in the following discussion.
We abbreviate the condition in Assmption \ref{ass:bd_assumption} as $\bm{b}<\bm{d}$.

For reformulating one-parameter persistence modules as flags, we first set an ambient free module as
\[
	F=\bigoplus_{i=1}^n \C[z]=\left\{\sum_{i=1}^n a_ie_i\colon a_i\in \C[z]\right\},
\]
where $\{e_i\colon i=1,\dots,n\}$ is a basis of $F$, and $F_1\subseteq F_0$ are regarded as submodules in $F$.

Now let us set a basis of $F_0$ as $F_0=\langle \xi_1,\dots,\xi_n\rangle$, where $\xi_i$'s are homogeneous with $\deg\xi_i=b_i$. The matrix formulation of $F_0$ in $F$ is given by
\begin{equation}\label{eq:mat-reform}
\begin{aligned}
\begin{pmatrix}\xi_1&\cdots&\xi_n\end{pmatrix}&=
\begin{pmatrix}e_1&\cdots&e_n\end{pmatrix}
\begin{pmatrix}
    \lambda_{11}z^{b_1} & \cdots & \lambda_{1n}z^{b_n}\\
	\vdots & & \vdots\\
	\lambda_{n1}z^{b_1}&\cdots&\lambda_{nn}z^{b_n}
\end{pmatrix}\\
&=\begin{pmatrix}e_1&\cdots&e_n\end{pmatrix}
\cdot\Lambda\cdot\diag(z^{b_1},\dots,z^{b_n}),
\end{aligned}
\end{equation}
where $\diag(a_1,\dots,a_n)$ expresses the diagonal matrix with $a_i$ at the $i$th entry.
Since $\langle \xi_1,\dots,\xi_n\rangle$ is a basis, the matrix $\Lambda=(\lambda_{ij})$ is an element of $\GL_n(\C)$.

Let us define a collection of
submodules of $F_0$ as 
\[
	\cK(\bm{b},\bm{d})=\cK:=\left\{K\subseteq F_0: K\cong \bigoplus_{i=1}^n z^{d_i}\C[z]\right\}.
\]
First, let us remark on the correspondence between base changes on submodules in $\cK(\bm{b},\bm{d})$ and flags in $\C^n$.
Let $K=\langle \eta_1,\dots,\eta_n \rangle=\langle \eta_1',\dots,\eta_n' \rangle$ be two expressions using different generators with $\deg(\eta_i)=\deg(\eta'_i)=d_i$ for $K\in \cK(\bm{b},\bm{d})$. 
The basis change matrix $C(z)$ for $(\eta'_1\dots \eta'_n)=(\eta_1\dots \eta_n)\cdot C(z)$ takes the following  upper triangular form:
\begin{equation}\label{eq:base_change}
\begin{aligned}
C(z)
&=\diag(z^{d_1},\dots,z^{d_n})^{-1}\cdot C\cdot \diag(z^{d_1},\dots,z^{d_n})\\
&=
\begin{pmatrix}
z^{-d_1}&&0\\
&\ddots&\\
0&&z^{-d_n}
\end{pmatrix}
\begin{pmatrix}
c_{11}&c_{12}&\cdots&c_{1n}\\
0&c_{22}&\cdots&c_{2n}\\
\vdots&\vdots&\ddots&\vdots\\
0&0&\cdots&c_{nn}
\end{pmatrix}
\begin{pmatrix}
z^{d_1}&&0\\
&\ddots&\\
0&&z^{d_n}
\end{pmatrix}.
%
\end{aligned}
\end{equation}
When we set an ordering of the basis of $F_0$ as $(\xi_1\cdots \xi_n)$, the submodule $K=\langle \eta_1,\dots,\eta_n\rangle\in \cK(\bm{b},\bm{d})$ can be represented as a matrix form
\begin{equation}\label{eq:birth-to-death-trans}
\begin{aligned}
	(\eta_1 \cdots \eta_n)
	&= (\xi_1\cdots \xi_n)
	\begin{pmatrix}
	a_{11}z^{d_1-b_1} & \dots & a_{1n}z^{d_n-b_1}\\
	\vdots & & \vdots\\
	a_{n1}z^{d_1-b_n}&\dots&a_{nn}z^{d_n-b_n}
	\end{pmatrix}\\
	&=(\xi_1\cdots \xi_n)\diag(z^{-b_1},\cdots,z^{-b_n})\cdot
	A\cdot \diag(z^{d_1},\dots,z^{d_n}),
\end{aligned}
\end{equation}
where $d_j\neq b_i$ for $a_{ij}\neq 0$ due to the minimality of the free resolution. We note $A=(a_{ij})\in \GL_n(\C)$. 
Using this formula and  (\ref{eq:base_change}), the basis change $(\eta_1'\dots \eta_n')=(\eta_1\dots \eta_n)C(z)$ becomes
\begin{equation}\label{eq:base-change}
\begin{aligned}
	(\eta_1'\dots \eta_n')&=(\eta_1\dots \eta_n)C(z)\\
	&=(\xi_1\cdots \xi_n)\diag(z^{-b_1},\cdots,z^{-b_n})\cdot A\cdot  C\cdot \diag(z^{d_1},\dots,z^{d_n}).
\end{aligned}
\end{equation}
Namely, the basis change of $K$ can be performed by acting an upper triangular matrix $C$ to the matrix representation $A$.


\begin{prop} \label{lem:submodule_flag_correspondence}
Assume $\bm{b}<\bm{d}$. Then, there exists a bijective correspondence:
\[
	\ca{K}(\bm{b},\bm{d})\cong \GL_n(\C)/B \cong \cF_n.
\]
\end{prop}

\begin{proof}
The second bijection is already shown in Section \ref{sec:flag}, so we will show the first bijection. 
Given a matrix $A=(a_{ij})\in \GL_n(\C)$, we want to associate it with an element in $\ca{K}(\bm{b},\bm{d})$. Death generators $\eta_1,\dots,\eta_n$ can be obtained from the birth generators via
\begin{equation}
	(\eta_1 \cdots \eta_n)= (\xi_1\cdots \xi_n)
	\begin{pmatrix}
	a_{11}z^{d_1-b_1} & \dots & a_{1n}z^{d_n-b_1}\\
	\vdots & & \vdots\\
	a_{n1}z^{d_1-b_n}&\dots&a_{nn}z^{d_n-b_n}
	\end{pmatrix}.
\end{equation}

Since $\deg\xi_i=b_i$, it is easy to verify $\deg \eta_i=d_i$ for $i=1,\dots,n$. Because $\{\xi_1,\dots, \xi_n\}$ is a basis of $F_0$ and $A\in \GL_n(\C)$, $\{\eta_1,\dots,\eta_n\}$ is $\C[z]$-linearly independent. Define a submodule $K_A:=\langle \eta_1,\dots,\eta_n\rangle\subset F_0$. This operation gives us a map
\[
\begin{aligned}
\varphi\colon \GL_n(\C)&\to\ca{K}(\bm{b},\bm{d})\\
A&\mapsto K_A
\end{aligned}.
\]

Given any submodule $K\in \ca{K}(\bm{b},\bm{d})$, as $K$ is free, we can choose a basis $K=\langle \eta_1,\dots,\eta_n\rangle$ and construct a matrix $A$ as in equation (\ref{eq:birth-to-death-trans}), implying that 
$\varphi$ is surjective. 
Then, the map $\varphi$ induces an equivalence relation $\sim$  on $\GL_n(\C)$:
\[
\text{for any $A,\,A'\in\GL_n(\C)$, $A \sim A'$ if and only if $\varphi(A)=\varphi(A')$}.
\]
Thus $\GL_n(\C)/\sim$ is bijective to $\ca{K}(\bm{b},\bm{d})$. 
Let $A$ and $A'$ be such two equivalent matrices and $\{\eta_1,\dots,\eta_n\}$, $\{\eta'_1,\dots,\eta'_n\}$ be the corresponding two generators. By (\ref{eq:base-change}), they satisfy the relation

\[
	(\eta_1'\dots \eta_n')=(\xi_1\cdots \xi_n)\diag(z^{-b_1},\cdots,z^{-b_n})\cdot A\cdot  C\cdot \diag(z^{d_1},\dots,z^{d_n}),
\]
where the shape of $C$ is specified in  (\ref{eq:base_change}) by setting $z=1$. Comparing this equation with the defining equation of $\{\eta'_1,\dots,\eta'_n\}$:
\[
\begin{aligned}
	(\eta'_1 \cdots \eta'_n)
	&=(\xi_1\cdots \xi_n)\diag(z^{-b_1},\cdots,z^{-b_n})\cdot
	A'\cdot \diag(z^{d_1},\dots,z^{d_n})
\end{aligned}
\]
implies  $A\cdot C=A'$ for some $C$. It is easy to verify that the collection of all valid shapes of $C$ is exactly $B$, from which the bijection follows.
\end{proof}

We recall that $\pi$ denotes the surjection $\pi: \GL_n(\C)\to\cF_n$.
Combining with the Schubert cell decomposition of the flag variety in  (\ref{eq:schubert_decomposition}), it follows that $\cK(\bm{b},\bm{d})$ has a cell decomposition:

\begin{align}\label{eq:partition_on_cK}
	\cK(\bm{b},\bm{d})=\bigsqcup_{w\in S_n}\cK^\circ_w(\bm{b},\bm{d}),\quad \cK^\circ_w(\bm{b},\bm{d})=\varphi\circ \pi^{-1}(X^\circ_w).
\end{align}

\begin{thm}\label{thm:pd_schubert}
The Schubert cell decomposition above gives a parametrization of isomorphism classes of  persistence modules with the birth sequence $\bm{b}$ and death sequence $\bm{d}$. 
\end{thm}

\begin{proof}
The isomorphism class of a persistence module $M$ is determined by its minimal free resolution
\[
\begin{tikzcd}
0\arrow[r]& F_1 \arrow[r,"\partial_1"]& F_0 \arrow[r,"\epsilon"] & M\arrow[r]& 0,
\end{tikzcd}
\]
so $M\simeq F_0/K_A$ for some $K_A\in \cK(\bm{b},\bm{d})$. 
Since $K_A= K_{A'}$ if and only if $A\sim A'$, the isomorphism class $[M]$ uniquely assigns a Schubert cell $X^\circ_w$, where $w$ is determined by $A\in U^+_w$.
Conversely, it is easy to see that, for a given Schubert cell $X^\circ_w$, we can obtain an isomorphism class of a persistence module corresponding to that Schubert cell.
\end{proof}


From this theorem, for a persistence module $M\simeq F_0/K_A$ with $A\in U^+_w$, 
we obtain the interval decomposition:
\[
M\simeq  \bigoplus^n_{i=1}
(z^{b_{w(i)}})/(z^{d_i}).
\]
which subsequently determines a persistence diagram
\[
	D^\circ_w(\bm{b},\bm{d})=\{(b_{w(i)},d_i)\colon i=1,\dots,n\}.
\]
Theorem \ref{thm:pd_schubert} and  (\ref{eq:partition_on_cK}) show the bijective correspondence among the Schubert cells $X^\circ_w$,  $\cK^\circ_w(\bm{b},\bm{d})$, and persistence diagrams $D^\circ_w(\bm{b},\bm{d})$ for given birth and death sequences $\bm{b}$ and $\bm{d}$. 
We also note that the set of persistence diagrams can be regarded as a poset by the Bruhat order.

\section{Schubert calculus on persistence modules}
\label{sec:calculus}
In this section, we first briefly review the formulation of the Chow ring and Schubert polynomials by following \cite{fulton, ikeda}. Based on this preparation, we introduce a multiplication on persistent diagrams in Section  \ref{sec:multiplication_on_ph}. 

\subsection{Chow ring}\label{sec:Chow_ring}
The length of a permutation $w\in S_n$ is given by $\ell(w)=\#\{(i,j)\in [n]\times [n]~:~i<j, w(i)>w(j)\}$.
An element $w_0\in S_n$ given by
\[
w_0(i)=n-i+1,\quad i=1,\dots,n
\]
is the unique permutation having the longest length $\ell(w_0)=\frac{n(n-1)}{2}$. For $w\in S_n$, we also define $w^\lor:=w_0w$ (i.e., $i\mapsto n-w(i)+1$).

We assign a formal symbol $\sigma_w$, called a Schubert class, to each $w\in S_n$ and consider the following additive groups
\begin{equation}    A^k(\cF_n):=\bigoplus_{w\in S_n,\ell(w)=k}\Z\sigma_w,\quad
    A^{*}(\cF_n):= \bigoplus_{k=0}^{\frac{n(n-1)}{2}}A^{k}(\cF_n).
\end{equation}
We first recall the following theorem
\begin{thm}{\rm (\cite{kleiman1974transversality}\label{thm:kleiman})}
Let $Y\subseteq \cF_n$ be an irreducible variety of codimension $k$, and $w\in S_n$ be a permutation with length $\ell(w)=k$. Then for a generic $g\in \GL_n(\C)$,
\[
\#(Y\cap g X_w)
\]
is a finite number irrelevant to the $g$ chosen.
\end{thm}

Following the same notation in Theorem \ref{thm:kleiman}, we define the symbol of an irreducible variety $Y$ by
\begin{equation}\label{eq:fundamental-cl}
[Y]:=\sum_{w\in S_n,\ell(w)=k}\#(Y\cap gX_w)\sigma_{w}\in A^k(\cF_n).
\end{equation}
Then, we can show $[\Omega_w]=\sigma_w$ and $[X_w]=\sigma_{w^\lor}$.
Theorem \ref{thm:kleiman} also guarantees that, for $u,v,w\in S_n$ with $\ell(u)=\ell(v)+\ell(w)$, the number
\begin{equation}\label{eq:fundamental-cl-2}
c^u_{wv}:=\#(
\Omega_w\cap g\Omega_v \cap X_u
)
\end{equation}
is a finite number irrelevant to the generic $g$ chosen. 

The above argument leads to the definition of the product of Schubert classes by
\[
\sigma_w\cdot\sigma_v=
\sum_{\substack{u\in S_n\\ \ell(u)=\ell(w)+\ell(v)}}c^u_{wv}\sigma_{u}.
\]
The additive group $A^*(\cF_n)$ equipped with this product is called the Chow ring. 
The coefficients $c^u_{wv}$ are called \textit{Littlewood–Richardson coefficients}. There exists no known rule to compute them in general cases, and limited rules are discovered only in several special cases. For example, when one of the multipliers is of form $\sigma_{r_i}$ for a transposition $r_i$ (only exchanging $i \leftrightarrow i+1$), the coefficients are determined by Monk's formula.

\begin{thm}{\rm (\cite{monk1959geometry}\label{thm:Monk_finite})}
Let $t_{ab}\in S_n$ denote the transposition between $a$ and $b$. For any $w\in S_n$ and $1\leq i<n$, 
\begin{equation}\label{eq:Monk}
\sigma_w\sigma_{r_i}=\sum_{\substack{a\leq i<b\\\ell(w t_{ab})=\ell(w)+1}}\sigma_{wt_{ab}}
\end{equation}
holds.
\end{thm}
Since the symmetric group $S_n$ is generated by $r_1,\dots,r_{n-1}$, we can compute any products of Schubert classes by the repeated applications of Monk's formula in principle.

Let $n$ be a positive integer. Recall that a polynomial $f\in \Z[z_1,\dots,z_n]$ is called a \textit{symmetric polynomial} if for any $w\in S_n$, $f(z_{w(1)},\dots,z_{w(n)})=f(z_1,\dots,z_n)$. In particular, we define the \textit{elementary symmetric polynomial} of degree $k$  to be 
\[
e_k=\sum_{1\leq i_1<\cdots<i_k\leq n}z_{i_1}\cdots z_{i_k}.
\]
Let us consider an ideal generated by elementary symmetric polynomials up to degree $n$:
\[
I_n:= \langle e_k~|~1\leq k\leq n\rangle.
\]
Then, we can identify the Chow ring $A^*(\cF_n)$ with the following quotient ring.

\begin{thm}\label{thm:ring_iso}
{\rm (\cite{fulton})}
The Chow ring $A^*(\cF_n)$ is isomorphic to the following graded ring
\[
\cR_n:= \Z[z_1,\dots,z_n]/I_n.
\]
\end{thm}

\subsection{Schubert polynomial}\label{sec:Schubert_polynomial}

For $m>n\in \N$, we define an embedding $\phi_{nm}: S_n\rightarrow S_m$ by
$\phi_{nm}(w)=w'$, where
\[
w'(i)=\left\{
\begin{array}{ll}
w(i),& i=1,\dots,n\\
i,& i=n+1,\dots,m.
\end{array}
\right.
\]
Then $(S_n,\phi_{nm})$ is a direct system over $\N$, leading to the infinite symmetric group
\[
S_{\infty}\coloneqq \lim\limits_{\substack{\longrightarrow\\n}} S_n=\bigsqcup_{n}S_n/\sim
\]
as the direct limit. For a permutation $w\in S_\infty$, we write $w\in S_n$ if $w$ fixes all numbers after $n$.

We also note that a natural embedding 
\[
\begin{aligned}
\iota_n: \C^n &\hookrightarrow \C^{n+1}\\
(x_1,\dots,x_n) &\mapsto (x_1,\dots,x_n,0)
\end{aligned}
\]
induces an embedding $\iota_n: \cF_n\to \cF_{n+1}$, so any flag can be regarded as a flag in a space with a higher dimension. 
Similarly, we define a composition
\[
\iota_{nm}:=\iota_{m-1}\circ\cdots\circ \iota_n\colon \cF_n \to \cF_m.
\]
The pullback of this map induces a morphism between Chow rings:
\[
\iota_{nm}^*\colon A^*(\cF_m)\to A^*(\cF_n).
\]
Then, this naturally defines the inverse limit
\begin{eqnarray*}
A^*(\cF_\infty)&:=&\lim_{\substack{\longleftarrow\\n}}A^*(\cF_n)\\
&=&\bigoplus_{r=0}^\infty\lim_{\substack{\longleftarrow\\n}} A^r(\cF_n).
\end{eqnarray*}

Let us denote Schubert classes in $A^*(\cF_n)$ as $\sigma_w^{(n)}$ for $w\in S_n$ by specifying $n$. 
Then, it follows from the fact (e.g., Proposition 10.3 in \cite{ikeda})
\[
\iota^*_n(\sigma_w^{(n+1)})=\left\{
\begin{array}{ll}
\sigma_w^{(n)},&w\in S_n\\
0,&w\notin S_n
\end{array}
\right.
\]
that the inverse limit can be written as
\begin{eqnarray*}
A^*(\cF_\infty)&=&\bigoplus_{r=0}^\infty\bigoplus_{\substack{w\in S_\infty\\\ell(w)=r}} \Z\hat{\sigma}_w\\
&=&\bigoplus_{w\in S_\infty}\Z\hat{\sigma}_w,
\end{eqnarray*}
where 
\[
\hat{\sigma}_w=(\sigma_w^{(n)})\in \lim_{\substack{\longleftarrow\\n}}A^r(\cF_n)
\]
denotes the limit of the Schubert class for $r=\ell(w)$.

Now, let us construct a polynomial ring with infinitely many variables. We note that any polynomial in $\Z[z_1,\dots,z_n]$ can be naturally embedded in $\Z[z_1,\dots,z_m]$ for $m\geq n$. Hence, we obtain a directed system whose direct limit
\[
\Z[z]:= \lim_{\substack{\longrightarrow\\n}}\Z[z_1,\dots,z_n] 
\]
is the polynomial ring in which we introduce Schubert polynomials below. 

\begin{thm}{\rm (\cite{fulton, ikeda})}
The map
\begin{align*}
 \Psi\colon \Z[z] &\to A^*(\cF_{\infty})=\bigoplus_{r=0}^{\infty}\lim_{\substack{\longleftarrow\\n}}A^r(\cF_n)\\
    f_r(z) &\mapsto (f_r^{(n)}~{\rm mod}~I_n),
\end{align*}
is bijective, where $f_r(z)\in\Z[z]$ is a homogeneous polynomial with degree $r$ and $f_r^{(n)}= f_r(z_1,\dots,z_n,0,\dots)$. Here, the ring structure is induced by the isomorphism in Theorem \ref{thm:ring_iso}.
\end{thm}

\begin{defn}
The Schubert polynomial $\mfkS_w$ for $w\in S_\infty$ is defined by
\[
\Psi(\mfkS_w)=\hat{\sigma}_w.
\]
\end{defn}

The Schubert polynomials can explicitly be described by using the difference operator:
\[
\partial_if:=\frac{f-r_if}{z_i-z_{i+1}}
\]
defined on $f\in \Z[z]$. Let us define $\delta(n)=(n-1,\dots,1,0)$. 
\begin{prop}{\rm (\cite{fulton, ikeda})}
Suppose $w\in S_\infty$ is expressed in $w\in S_n$ by $w=w_0r_{i_1}\dots r_{i_k}$ with $k=\ell(w_0)-\ell(w)$, where $w_0$ is the longest element in $S_n$. 
Then the polynomial
\[
\partial_{i_k}\dots\partial_{i_1}z^{\delta(n)}
\]
is defined independently of the choice of $n$ and $(i_1,\dots,i_k)$, and coincides with the Schubert polynomial $\mfkS_w$.
\end{prop}

For later use, we show the Hasse diagrams and Schubert polynomials for $S_2$ and $S_3$.
\begin{figure}\label{fig:Hasse}
\centering
\begin{tikzcd}[column sep=tiny,row sep=normal]
    {(2\ 1)} \arrow[d,no head] \arrow[r, phantom, "z_1" near start] & {} \\
    {(1\ 2)} \arrow[r, phantom, "1" near start] & {}
\end{tikzcd}
\hspace{1cm}
\begin{tikzcd}[column sep=tiny,row sep=normal]
    {} 
    & 
    {} 
    &
    {(3\ 2\ 1)} 
    \arrow[ld, no head] 
    \arrow[rd, no head] 
    \arrow[r, phantom, "z_1^2z_2" near start] 
    & 
    {}
    &
    {}
    \\
    {}
    &
    {(3\ 1\ 2)} 
    \arrow[d, no head] 
    \arrow[l, phantom, "z_1^2" near start] 
    & 
    {}
    & 
    {(2\ 3\ 1)} 
    \arrow[lld, no head] 
    \arrow[d, no head] 
    \arrow[r, phantom, "z_1z_2"{pos=1.7}] 
    &
    {}
    \\
    {}
    &
    {(2\ 1\ 3)} 
    \arrow[rd, no head] 
    \arrow[l, phantom, "z_1" near start] 
    & 
    {}
    & 
    {(1\ 3\ 2)} 
    \arrow[from=llu, crossing over, no head] 
    \arrow[ld, no head] 
    \arrow[r, phantom, "z_1+z_2"{pos=1.9}] 
    &
    {}
    \\
    {}
    &
    {}
    & 
    {(1\ 2\ 3)} \arrow[r, phantom, "1"{pos=0.2}] &
    {}
    &
    {}
\end{tikzcd}
\caption{Hasse diagram and Schubert polynomials for $S_2$ (left) and $S_3$(right).}
\end{figure}
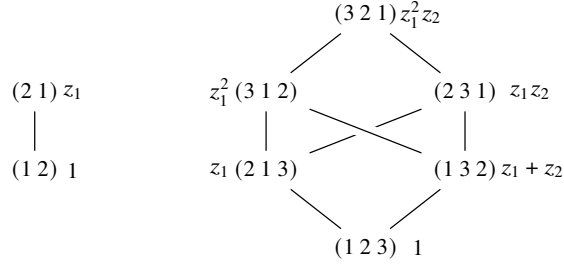

Similar to Theorem \ref{thm:Monk_finite}, the Schubert polynomials also satisfy the Monk formula:
\begin{thm}\label{thm:Monk_infinite}
For $w\in S_\infty$ and $i\geq 1$, 
\[
\mfkS_w\mfkS_{r_i}=\sum_{\substack{a\leq i<b\\\ell(w t_{ab})=\ell(w)+1}}\mfkS_{wt_{ab}}.
\]
\end{thm}

\subsection{Multiplication on persistence diagrams}\label{sec:multiplication_on_ph}

In Section \ref{sec:Chow_ring} and \ref{sec:Schubert_polynomial}, we defined the product structure on Schubert classes (and Schubert polynomials) which bijectively correspond to persistence diagrams
through flag varieties. This means that we can naturally introduce a product structure on persistence diagrams, which has a natural geometric meaning induced by that of Schubert varieties, i.e., the intersection of two isomorphism classes of persistence modules in $\cK(\bm{b},\bm{d})$. In this section, we show several examples of the products of persistence diagrams.

From Theorem \ref{thm:pd_schubert}, we have the correspondence between the Schubert variety $X^\circ_w$ and the persistence diagram 
\[
	D^\circ_w(\bm{b},\bm{d})=\{(b_{w(i)},d_i)\colon i=1,\dots,n\}
\]
for each $w\in S_n$.
In parallel to the definition $X_w=\sqcup_{u\leq w}X^\circ_u$, we introduce a symbol
\[
D_w(\bm{b},\bm{d})=\bigsqcup_{u\leq w}D_u^\circ(\bm{b},\bm{d}). 
\]
This definition is reasonable since $D_w(\bm{b},\bm{d})$ contains the information of degenerate cases. 
Also, we recall $[X_w]=\sigma_{w^\lor}.$
From these correspondences, we naturally identify the persistence diagrams $D_w(\bm{b},\bm{d})$ with $\sigma_{w^\lor}$ from now on. Similarly, we also associate $\hat{D}_w(\bm{b},\bm{d})$ with $\hat{\sigma}_w$ as a limiting object $n\rightarrow \infty$. In the following, we simply abbreviate $D_w=D_w(\bm{b},\bm{d})$. We remark that $D_{w_0}$ becomes the identity element of the product. 

\begin{ex}
For $n=2$, let us consider the persistence diagram $D_w({\bm{b},\bm{d}})$ with $w=(12)\in S_2$ (see Figure \ref{fig:pd_12}) and the product $D_w\cdot D_w$. Expressing the product in the Schubert classes lead to
\[
\sigma_{(12)^\lor}\cdot\sigma_{(12)^\lor}=\sigma_{(21)}\cdot\sigma_{(21)}.
\]
We apply Monk's formula (\ref{eq:Monk}) with $w=r_1=(21)$. In this case, the only choice for $(a,b)=(1,2)$ gives  $\ell(wt_{12})\neq \ell(w)+1$.  Therefore, 
\[
\sigma_{(12)^\lor}\cdot\sigma_{(12)^\lor}=0,
\]
meaning $D_w\cdot D_w=0$. 

On the other hand, the same product using the Schubert polynomials leads to 
\[
\mfkS_{(12)^\lor}\cdot\mfkS_{(12)^\lor}=
\mfkS_{(21)}\cdot\mfkS_{(21)}=z_1^2=
\mfkS_{(312)}=\mfkS_{(132)^\lor}.
\]
The reason for this discrepancy is caused by the fact that Monk's formula in Theorem \ref{thm:Monk_finite} fixes the total dimension $n$, while that in Theorem \ref{thm:Monk_infinite} can consider $n$ to be large enough. More explicitly, the Schubert class $\sigma_{(12)^\lor}$ expresses a flag with codimension one in $\C^2$, and then $\sigma_{(12)^\lor}\cdot \sigma_{(12)^\lor}$ corresponds to the intersection with codimension two, leading to zero in the product. On the other hand, in the setting of Schubert polynomials, we embed $w$ in $S_3$ (or $S_n$ for $n>2$), whereby $w$ is understood as $w(213)$. Then, the total space becomes $\C^3$, and hence the codimension two intersection can be nontrivial. As a result, we have $\hat{D}_w\cdot\hat{D}_w=\hat{D}_{(132)}$.

\end{ex}

\begin{ex}
For $n=3$, let us consider the product $D_{(213)}\cdot D_{(231)}$. From 
$
\sigma_{(213)^\lor}\cdot\sigma_{(231)^\lor}=\sigma_{(231)}\cdot\sigma_{(213)}
$, let us set $w=(231)$ and $r_1=(213)$. Then, for $(a,b)=(1,2),(1,3)$ in the Monk's formula (\ref{eq:Monk}), we have $\ell(wt_{12})=\ell(w)+1$ and $\ell(wt_{13})\neq \ell(w)+1$. Thus 
\[
\sigma_{(213)^\lor}\cdot\sigma_{(231)^\lor}=\sigma_{(321)}=\sigma_{(123)^\lor}.
\]
We also have the same result using the Schubert polynomials:
\[
\mfkS_{(213)^\lor}\cdot \mfkS_{(231)^\lor}=
\mfkS_{(231)}\cdot \mfkS_{(213)}=z_1z_2\cdot z_1
=\mfkS_{(321)}=\mfkS_{(123)^\lor}.
\]
Therefore, we have $D_{(213)}\cdot D_{(231)} =\hat{D}_{(213)}\cdot \hat{D}_{(231)}=D_{(123)}$.

Similarly, we can show that 
\[
D_{(231)}\cdot D_{(312)}=
\hat{D}_{(231)}\cdot \hat{D}_{(312)}=D_{(213)}+D_{(132)}.
\]
\begin{figure}[htbp]
\begin{subfigure}[t]{0.92\textwidth}
\[
\begin{aligned}
&\begin{tikzpicture}[scale=0.52]
\newcommand\step{0.5}
     \draw [ultra thick] (0.5,0)--(6.5,0);
     \node [below] at (1,0) {$b_1$};
     \node [below] at (2,0) {$b_2$};
     \node [below] at (3,0) {$b_3$};
     \node [below] at (4,0) {$d_1$};
     \node [below] at (5,0) {$d_2$};
     \node [below] at (6,0) {$d_3$};
     \draw [thick] (3,0) -- (3,3*\step);
     \draw [thick] (3,3*\step) -- (6,3*\step);
     \draw [thick] (6,3*\step) -- (6,0);
     \draw [thick] (2,0) -- (2,1*\step);
     \draw [thick] (2,1*\step) -- (4,1*\step);
     \draw [thick] (4,1*\step) -- (4,0);
     \draw [thick] (1,0) -- (1,2*\step);
     \draw [thick] (1,2*\step) -- (5,2*\step);
     \draw [thick] (5,2*\step) -- (5,0);
     \filldraw [cyan] (3,1*\step) circle [radius=.08];
     \filldraw [cyan] (3,2*\step) circle [radius=.08];
\end{tikzpicture}
\;
\begin{tikzpicture}[scale=0.52]
\draw (0,0)--(0,0);
\filldraw [black] (0,1.3) circle [radius=0.05];
\end{tikzpicture}
\;
\begin{tikzpicture}[scale=0.52]
\newcommand\step{0.5}
     \draw [ultra thick] (0.5,0)--(6.5,0);
     \node [below] at (1,0) {$b_1$};
     \node [below] at (2,0) {$b_2$};
     \node [below] at (3,0) {$b_3$};
     \node [below] at (4,0) {$d_1$};
     \node [below] at (5,0) {$d_2$};
     \node [below] at (6,0) {$d_3$};
     \draw [thick] (3,0) -- (3,2*\step);
     \draw [thick] (3,2*\step) -- (5,2*\step);
     \draw [thick] (5,2*\step) -- (5,0);
     \draw [thick] (2,0) -- (2,1*\step);
     \draw [thick] (2,1*\step) -- (4,1*\step);
     \draw [thick] (4,1*\step) -- (4,0);
     \draw [thick] (1,0) -- (1,3*\step);
     \draw [thick] (1,3*\step) -- (6,3*\step);
     \draw [thick] (6,3*\step) -- (6,0);
     \filldraw [cyan] (3,1*\step) circle [radius=.08];
\end{tikzpicture}
\;
\begin{tikzpicture}[scale=0.52]
\draw (0,0)--(0,0);
\node [below] at (0,1.3) {$=$};
\end{tikzpicture}
\;
\begin{tikzpicture}[scale=0.52]
\newcommand\step{0.5}
     \draw [ultra thick] (0.5,0)--(6.5,0);
     \node [below] at (1,0) {$b_1$};
     \node [below] at (2,0) {$b_2$};
     \node [below] at (3,0) {$b_3$};
     \node [below] at (4,0) {$d_1$};
     \node [below] at (5,0) {$d_2$};
     \node [below] at (6,0) {$d_3$};
     \draw [thick] (3,0) -- (3,3*\step);
     \draw [thick] (3,3*\step) -- (6,3*\step);
     \draw [thick] (6,3*\step) -- (6,0);
     \draw [thick] (2,0) -- (2,2*\step);
     \draw [thick] (2,2*\step) -- (5,2*\step);
     \draw [thick] (5,2*\step) -- (5,0);
     \draw [thick] (1,0) -- (1,1*\step);
     \draw [thick] (1,1*\step) -- (4,1*\step);
     \draw [thick] (4,1*\step) -- (4,0);
     \filldraw [cyan] (3,1*\step) circle [radius=.08];
     \filldraw [cyan] (3,2*\step) circle [radius=.08];
     \filldraw [cyan] (2,1*\step) circle [radius=.08];
\end{tikzpicture}
\end{aligned}
\]
\caption{$D_{(213)}\cdot D_{(231)}$}
\label{fig:first}
\end{subfigure}
\begin{subfigure}[t]{\textwidth}
\[
\begin{aligned}
&\begin{tikzpicture}[scale=0.52]
\newcommand\step{0.5}
     \draw [ultra thick] (0.5,0)--(6.5,0);
     \node [below] at (1,0) {$b_1$};
     \node [below] at (2,0) {$b_2$};
     \node [below] at (3,0) {$b_3$};
     \node [below] at (4,0) {$d_1$};
     \node [below] at (5,0) {$d_2$};
     \node [below] at (6,0) {$d_3$};
     \draw [thick] (3,0) -- (3,2*\step);
     \draw [thick] (3,2*\step) -- (5,2*\step);
     \draw [thick] (5,2*\step) -- (5,0);
     \draw [thick] (2,0) -- (2,1*\step);
     \draw [thick] (2,1*\step) -- (4,1*\step);
     \draw [thick] (4,1*\step) -- (4,0);
     \draw [thick] (1,0) -- (1,3*\step);
     \draw [thick] (1,3*\step) -- (6,3*\step);
     \draw [thick] (6,3*\step) -- (6,0);
     \filldraw [cyan] (3,1*\step) circle [radius=.08];
\end{tikzpicture}
\;
\begin{tikzpicture}[scale=0.52]
\draw (0,0)--(0,0);
\filldraw [black] (0,1.8) circle [radius=0.05];
\end{tikzpicture}
\;
\begin{tikzpicture}[scale=0.52]
\newcommand\step{0.5}
     \draw [ultra thick] (0.5,0)--(6.5,0);
     \node [below] at (1,0) {$b_1$};
     \node [below] at (2,0) {$b_2$};
     \node [below] at (3,0) {$b_3$};
     \node [below] at (4,0) {$d_1$};
     \node [below] at (5,0) {$d_2$};
     \node [below] at (6,0) {$d_3$};
     \draw [thick] (3,0) -- (3,1*\step);
     \draw [thick] (3,1*\step) -- (4,1*\step);
     \draw [thick] (4,1*\step) -- (4,0);
     \draw [thick] (2,0) -- (2,3*\step);
     \draw [thick] (2,3*\step) -- (6,3*\step);
     \draw [thick] (6,3*\step) -- (6,0);
     \draw [thick] (1,0) -- (1,2*\step);
     \draw [thick] (1,2*\step) -- (5,2*\step);
     \draw [thick] (5,2*\step) -- (5,0);
     \filldraw [cyan] (2,2*\step) circle [radius=.08];
\end{tikzpicture}
\;
\begin{tikzpicture}[scale=0.52]
\draw (0,0)--(0,0);
\node [below] at (0,1.8) {$=$};
\end{tikzpicture}
\;
\begin{tikzpicture}[scale=0.52]
\newcommand\step{0.5}
     \draw [ultra thick] (0.5,0)--(6.5,0);
     \node [below] at (1,0) {$b_1$};
     \node [below] at (2,0) {$b_2$};
     \node [below] at (3,0) {$b_3$};
     \node [below] at (4,0) {$d_1$};
     \node [below] at (5,0) {$d_2$};
     \node [below] at (6,0) {$d_3$};
     \draw [thick] (3,0) -- (3,3*\step);
     \draw [thick] (3,3*\step) -- (6,3*\step);
     \draw [thick] (6,3*\step) -- (6,0);
     \draw [thick] (2,0) -- (2,1*\step);
     \draw [thick] (2,1*\step) -- (4,1*\step);
     \draw [thick] (4,1*\step) -- (4,0);
     \draw [thick] (1,0) -- (1,2*\step);
     \draw [thick] (1,2*\step) -- (5,2*\step);
     \draw [thick] (5,2*\step) -- (5,0);
     \filldraw [cyan] (3,1*\step) circle [radius=.08];
     \filldraw [cyan] (3,2*\step) circle [radius=.08];
\end{tikzpicture}
\;
\begin{tikzpicture}[scale=0.52]
\draw (0,0)--(0,0);
\node  at (0,1.8) {$+$};
\end{tikzpicture}
\;
\begin{tikzpicture}[scale=0.52]
\newcommand\step{0.5}
     \draw [ultra thick] (0.5,0)--(6.5,0);
     \node [below] at (1,0) {$b_1$};
     \node [below] at (2,0) {$b_2$};
     \node [below] at (3,0) {$b_3$};
     \node [below] at (4,0) {$d_1$};
     \node [below] at (5,0) {$d_2$};
     \node [below] at (6,0) {$d_3$};
     \draw [thick] (3,0) -- (3,2*\step);
     \draw [thick] (3,2*\step) -- (5,2*\step);
     \draw [thick] (5,2*\step) -- (5,0);
     \draw [thick] (2,0) -- (2,3*\step);
     \draw [thick] (2,3*\step) -- (6,3*\step);
     \draw [thick] (6,3*\step) -- (6,0);
     \draw [thick] (1,0) -- (1,1*\step);
     \draw [thick] (1,1*\step) -- (4,1*\step);
     \draw [thick] (4,1*\step) -- (4,0);
     \filldraw [cyan] (2,1*\step) circle [radius=.08];
     \filldraw [cyan] (3,1*\step) circle [radius=.08];
\end{tikzpicture}
\end{aligned}
\]
\caption{$D_{(231)}\cdot D_{(312)}$}
\label{fig:second}
\end{subfigure}
\caption{Illustration of $D_{(213)}\cdot D_{(231)}$ and $D_{(231)}\cdot D_{(312)}$}
\label{fig:pd_12}
\end{figure}
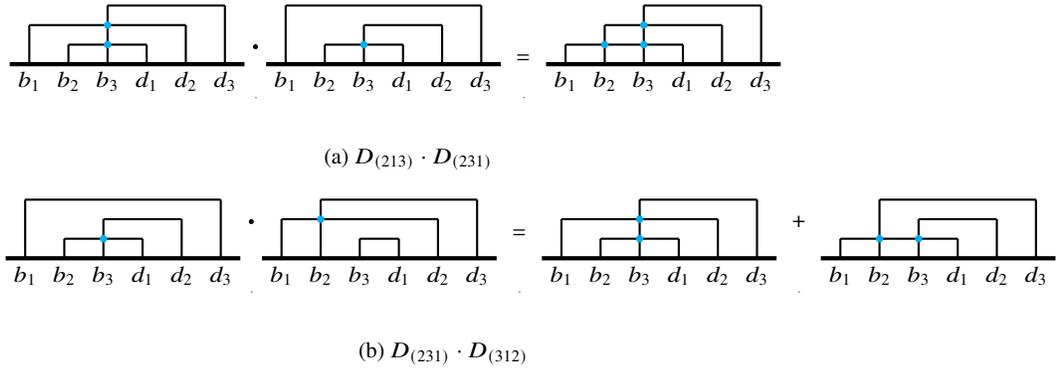
\end{ex}

\begin{ex}
For $n=4$, the following product can be calculated in a similar way as above:
\[
D_{(3241)}\cdot D_{(2413)}=D_{(1324)}+D_{(1243)}.
\]
\end{ex}

Using the same settings as above, this relation can be graphically expressed as
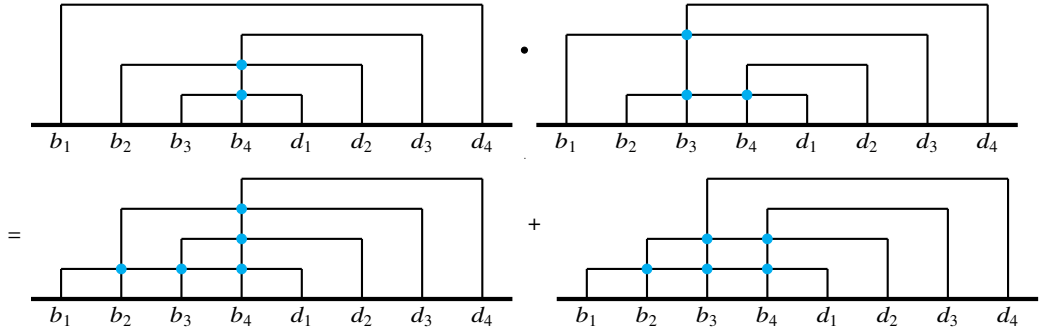
\begin{figure}[H]
    \centering
    \[
\begin{aligned}
&\begin{tikzpicture}[scale=0.8]
\newcommand\step{0.5}
     \draw [ultra thick] (0.5,0)--(8.5,0);
     \node [below] at (1,0) {$b_1$};
     \node [below] at (2,0) {$b_2$};
     \node [below] at (3,0) {$b_3$};
     \node [below] at (4,0) {$b_4$};
     \node [below] at (5,0) {$d_1$};
     \node [below] at (6,0) {$d_2$};
     \node [below] at (7,0) {$d_3$};
     \node [below] at (8,0) {$d_4$};
     \draw [thick] (3,0) -- (3,1*\step);
     \draw [thick] (3,1*\step) -- (5,1*\step);
     \draw [thick] (5,1*\step) -- (5,0);
     \draw [thick] (2,0) -- (2,2*\step);
     \draw [thick] (2,2*\step) -- (6,2*\step);
     \draw [thick] (6,2*\step) -- (6,0);
     \draw [thick] (4,0) -- (4,3*\step);
     \draw [thick] (4,3*\step) -- (7,3*\step);
     \draw [thick] (7,3*\step) -- (7,0);
     \draw [thick] (1,0) -- (1,4*\step);
     \draw [thick] (1,4*\step) -- (8,4*\step);
     \draw [thick] (8,4*\step) -- (8,0);
     \filldraw [cyan] (4,1*\step) circle [radius=.08];
     \filldraw [cyan] (4,2*\step) circle [radius=.08];
\end{tikzpicture}
\;
\begin{tikzpicture}[scale=0.8]
\draw (0,0)--(0,0);
\filldraw [black] (0,1.8) circle [radius=0.05];
\end{tikzpicture}
\;
\begin{tikzpicture}[scale=0.8]
\newcommand\step{0.5}
     \draw [ultra thick] (0.5,0)--(8.5,0);
     \node [below] at (1,0) {$b_1$};
     \node [below] at (2,0) {$b_2$};
     \node [below] at (3,0) {$b_3$};
     \node [below] at (4,0) {$b_4$};
     \node [below] at (5,0) {$d_1$};
     \node [below] at (6,0) {$d_2$};
     \node [below] at (7,0) {$d_3$};
     \node [below] at (8,0) {$d_4$};
     \draw [thick] (2,0) -- (2,1*\step);
     \draw [thick] (2,1*\step) -- (5,1*\step);
     \draw [thick] (5,1*\step) -- (5,0);
     \draw [thick] (4,0) -- (4,2*\step);
     \draw [thick] (4,2*\step) -- (6,2*\step);
     \draw [thick] (6,2*\step) -- (6,0);
     \draw [thick] (1,0) -- (1,3*\step);
     \draw [thick] (1,3*\step) -- (7,3*\step);
     \draw [thick] (7,3*\step) -- (7,0);
     \draw [thick] (3,0) -- (3,4*\step);
     \draw [thick] (3,4*\step) -- (8,4*\step);
     \draw [thick] (8,4*\step) -- (8,0);
     \filldraw [cyan] (3,1*\step) circle [radius=.08];
     \filldraw [cyan] (3,3*\step) circle [radius=.08];
     \filldraw [cyan] (4,1*\step) circle [radius=.08];
\end{tikzpicture} \\
\begin{tikzpicture}[scale=0.8]
\draw (0,0)--(0,0);
\node [below] at (0,1.8) {$=$};
\end{tikzpicture}&
\begin{tikzpicture}[scale=0.8]
\newcommand\step{0.5}
     \draw [ultra thick] (0.5,0)--(8.5,0);
     \node [below] at (1,0) {$b_1$};
     \node [below] at (2,0) {$b_2$};
     \node [below] at (3,0) {$b_3$};
     \node [below] at (4,0) {$b_4$};
     \node [below] at (5,0) {$d_1$};
     \node [below] at (6,0) {$d_2$};
     \node [below] at (7,0) {$d_3$};
     \node [below] at (8,0) {$d_4$};
     \draw [thick] (1,0) -- (1,1*\step);
     \draw [thick] (1,1*\step) -- (5,1*\step);
     \draw [thick] (5,1*\step) -- (5,0);
     \draw [thick] (3,0) -- (3,2*\step);
     \draw [thick] (3,2*\step) -- (6,2*\step);
     \draw [thick] (6,2*\step) -- (6,0);
     \draw [thick] (2,0) -- (2,3*\step);
     \draw [thick] (2,3*\step) -- (7,3*\step);
     \draw [thick] (7,3*\step) -- (7,0);
     \draw [thick] (4,0) -- (4,4*\step);
     \draw [thick] (4,4*\step) -- (8,4*\step);
     \draw [thick] (8,4*\step) -- (8,0);
     \filldraw [cyan] (2,1*\step) circle [radius=.08];
     \filldraw [cyan] (3,1*\step) circle [radius=.08];
     \filldraw [cyan] (4,1*\step) circle [radius=.08];
     \filldraw [cyan] (4,2*\step) circle [radius=.08];
     \filldraw [cyan] (4,3*\step) circle [radius=.08];
\end{tikzpicture}
\;
\begin{tikzpicture}[scale=0.8]
\draw (0,0)--(0,0);
\node  at (0,1.8) {$+$};
\end{tikzpicture}
\;
\begin{tikzpicture}[scale=0.8]
\newcommand\step{0.5}
     \draw [ultra thick] (0.5,0)--(8.5,0);
     \node [below] at (1,0) {$b_1$};
     \node [below] at (2,0) {$b_2$};
     \node [below] at (3,0) {$b_3$};
     \node [below] at (4,0) {$b_4$};
     \node [below] at (5,0) {$d_1$};
     \node [below] at (6,0) {$d_2$};
     \node [below] at (7,0) {$d_3$};
     \node [below] at (8,0) {$d_4$};
     \draw [thick] (1,0) -- (1,1*\step);
     \draw [thick] (1,1*\step) -- (5,1*\step);
     \draw [thick] (5,1*\step) -- (5,0);
     \draw [thick] (2,0) -- (2,2*\step);
     \draw [thick] (2,2*\step) -- (6,2*\step);
     \draw [thick] (6,2*\step) -- (6,0);
     \draw [thick] (4,0) -- (4,3*\step);
     \draw [thick] (4,3*\step) -- (7,3*\step);
     \draw [thick] (7,3*\step) -- (7,0);
     \draw [thick] (3,0) -- (3,4*\step);
     \draw [thick] (3,4*\step) -- (8,4*\step);
     \draw [thick] (8,4*\step) -- (8,0);
     \filldraw [cyan] (2,1*\step) circle [radius=.08];
     \filldraw [cyan] (3,1*\step) circle [radius=.08];
     \filldraw [cyan] (3,2*\step) circle [radius=.08];
     \filldraw [cyan] (4,1*\step) circle [radius=.08];
     \filldraw [cyan] (4,2*\step) circle [radius=.08];
\end{tikzpicture}
\end{aligned}
\]
    \caption{Illustration of $D_{(3241)}\cdot D_{(2413)}$}
    \label{fig:mult-PD-exa}
\end{figure}


%
%
%
%
%
\section{Discussion}
In this paper, we defined a multiplication on persistence diagrams by means of Schubert calculus. The meaning of this multiplication stems from algebro-geometric intersections of varieties of persistence modules. More precisely, by interpreting 
(\ref{eq:fundamental-cl}) and (\ref{eq:fundamental-cl-2}) from persistence modules viewpoints, this multiplication characterizes the intersection of {\it generic persistence modules} in a representation space. This paper only focuses on the algebraic side of the multiplication, but its meaning in topological data analysis, such as the multiplication of homological generators in data science, should also be addressed in future. Here, we remark that the paper \cite{curry} studies a genericity of merge trees in the context of inverse problems of persistence diagrams, so it may be useful to shed new light on the above problem.

\begin{acknowledgement}
The authors would like to express their sincere gratitude to Hiroyuki Ochiai for stimulating discussions. 
YH and KY are supported by the Japan Society for the Promotion of Science, Grant-in-Aid for Transformative Research Areas (A) (22A201).
\end{acknowledgement}

\printbibliography[heading=bibnonumbered, title=References]

\end{document}